\documentclass[12pt]{article}
\usepackage{amsmath, amsthm, amssymb}

\title{Simplest miniversal deformations of matrices,  matrix pencils, and contragredient matrix pencils\footnotetext{This is the authors' version of a work that was published in Linear Algebra Appl. 302--303 (1999) 45--61.}}
\author{M$^{\mbox{\b{a}}}$
  Isabel Garc\'\i a-Planas\\
Dept. de Matem\`atica Aplicada
I\\ Universitat Polit\`ecnica de
Catalunya\\ Marqu\'es de
Sentmenat 63-4-3, 08029
Barcelona, Spain\\
igarcia@ma1.upc.es\\ \\ Vladimir
V. Sergeichuk\thanks{Partially
supported by Grant No. UM1-314
of the U.S. Civilian Research
and Development Foundation for
the Independent States of the
Former Soviet Union.} \\
        Institute of Mathematics,
        Tereshchenkivska St. 3 \\
        Kiev, Ukraine \\
        sergeich@imath.kiev.ua}
\date{}
\begin{document}
\maketitle

\begin{abstract}
For a family of linear operators
$A(\vec\lambda): U\to U$ over
$\mathbb C$ that smoothly depend
on parameters $\vec\lambda=
(\lambda_1,\dots,\lambda_k)$, V.
I. Arnold obtained the simplest
normal form of their matrices
relative to a smoothly depending
on $\vec\lambda$ change of a
basis in $U$. We solve the same
problem for a family of linear
operators $A(\vec\lambda): U\to
U$ over $\mathbb R$, for a
family of pairs of linear
mappings $A(\vec\lambda): U\to
V,\ B(\vec\lambda): U\to V$ over
$\mathbb C$ and $\mathbb R$, and
for a family of pairs of counter
linear mappings $A(\vec\lambda):
U\to V,\ B(\vec\lambda): V\to U$
over $\mathbb C$ and $\mathbb
R$.
\end{abstract}

\def\newpic#1{%
   \def\emline##1##2##3##4##5##6{%
      \put(##1,##2){\special{em:point #1##3}}%
      \put(##4,##5){\special{em:point #1##6}}%
      \special{em:line #1##3,#1##6}}}
\newpic{}

\newcommand{\matr}[4]%
{\left[\genfrac{}{}{0pt}{}{#1}{#3}\,
\genfrac{}{}{0pt}{}{#2}{#4}\right]}

\renewcommand{\le}{\leqslant}
\renewcommand{\ge}{\geqslant}

\newtheorem{theorem}{Theorem}[section]
\newtheorem{lemma}{Lemma}[section]
\newtheorem{corollary}{Corollary}[section]

\theoremstyle{remark}
\newtheorem{remark}{Remark}[section]
\newtheorem{example}{Example}[section]

\newcommand{\quiva}{%
\special{em:linewidth 0.4pt}
\unitlength 0.3mm
\linethickness{0.4pt}
\begin{picture}(17.00,15.50)(6,7)
\put(3.00,10.00){\makebox(0,0)[cc]{$\cdot$}}
\emline{5.00}{11.00}{201}{8.10}{12.06}{202}
\emline{8.10}{12.06}{203}{10.74}{12.73}{204}
\emline{10.74}{12.73}{205}{12.92}{13.01}{206}
\emline{12.92}{13.01}{207}{14.63}{12.91}{208}
\emline{14.63}{12.91}{209}{15.88}{12.43}{2010}
\emline{15.88}{12.43}{2011}{16.67}{11.56}{2012}
\emline{16.67}{11.56}{2013}{17.00}{10.00}{2014}
\put(5.00,9.00){\vector(-3,1){0.2}}
\emline{17.00}{10.00}{2015}{16.77}{8.67}{2016}
\emline{16.77}{8.67}{2017}{16.07}{7.72}{2018}
\emline{16.07}{7.72}{2019}{14.92}{7.15}{2020}
\emline{14.92}{7.15}{2021}{13.30}{6.98}{2022}
\emline{13.30}{6.98}{2023}{11.21}{7.18}{2024}
\emline{11.21}{7.18}{2025}{8.67}{7.78}{2026}
\emline{8.67}{7.78}{2027}{5.00}{9.00}{2028}
\end{picture}
}

\newcommand{\quivb}{%
\special{em:linewidth 0.4pt}
\unitlength 0.3mm
\linethickness{0.4pt}
\begin{picture}(23.00,15.00)(6,7)
\put(3.00,10.00){\makebox(0,0)[cc]{$\cdot$}}
\put(23.00,10.00){\makebox(0,0)[cc]{$\cdot$}}
\put(21.00,11.00){\vector(3,-1){0.2}}
\emline{5.00}{11.00}{101}{7.22}{11.96}{102}
\emline{7.22}{11.96}{103}{9.44}{12.60}{104}
\emline{9.44}{12.60}{105}{11.67}{12.94}{106}
\emline{11.67}{12.94}{107}{13.89}{12.98}{108}
\emline{13.89}{12.98}{109}{16.11}{12.70}{1010}
\emline{16.11}{12.70}{1011}{18.33}{12.11}{1012}
\emline{18.33}{12.11}{1013}{21.00}{11.00}{1014}
\put(21.00,9.00){\vector(3,1){0.2}}
\emline{5.00}{9.00}{1015}{7.22}{8.04}{1016}
\emline{7.22}{8.04}{1017}{9.44}{7.40}{1018}
\emline{9.44}{7.40}{1019}{11.67}{7.06}{1020}
\emline{11.67}{7.06}{1021}{13.89}{7.02}{1022}
\emline{13.89}{7.02}{1023}{16.11}{7.30}{1024}
\emline{16.11}{7.30}{1025}{18.33}{7.89}{1026}
\emline{18.33}{7.89}{1027}{21.00}{9.00}{1028}
\end{picture}
}

\newcommand{\quivc}{%
\special{em:linewidth 0.4pt}
\unitlength 0.3mm
\linethickness{0.4pt}
\begin{picture}(23.00,15.00)(6,7)
\put(3.00,10.00){\makebox(0,0)[cc]{$\cdot$}}
\put(23.00,10.00){\makebox(0,0)[cc]{$\cdot$}}
\put(21.00,11.00){\vector(3,-1){0.2}}
\emline{5.00}{11.00}{301}{7.22}{11.96}{302}
\emline{7.22}{11.96}{303}{9.44}{12.60}{304}
\emline{9.44}{12.60}{305}{11.67}{12.94}{306}
\emline{11.67}{12.94}{307}{13.89}{12.98}{308}
\emline{13.89}{12.98}{309}{16.11}{12.70}{3010}
\emline{16.11}{12.70}{3011}{18.33}{12.11}{3012}
\emline{18.33}{12.11}{3013}{21.00}{11.00}{3014}
\put(5.00,9.00){\vector(-3,1){0.2}}
\emline{21.00}{9.00}{3015}{18.78}{8.04}{3016}
\emline{18.78}{8.04}{3017}{16.56}{7.40}{3018}
\emline{16.56}{7.40}{3019}{14.33}{7.06}{3020}
\emline{14.33}{7.06}{3021}{12.11}{7.02}{3022}
\emline{12.11}{7.02}{3023}{9.89}{7.30}{3024}
\emline{9.89}{7.30}{3025}{7.67}{7.89}{3026}
\emline{7.67}{7.89}{3027}{5.00}{9.00}{3028}
\end{picture}
}

\section{Introduction} \label{s1}

All matrices and representations
are considered over a field
$\mathbb F\in\{\mathbb C,
\mathbb R\}$. We base on ideas
and methods from Arnold's
article \cite{arn}, extending
them on quiver representations.

Systems of linear mappings are
conveniently studied if we
consider them as representations
of a quiver. A {\it quiver} is a
directed graph, its {\it
representation $A$} over
$\mathbb F$ is given by
assigning to each vertex $i$ a
finite dimensional vector space
$A_i$ over $\mathbb F$ and to
each arrow $\alpha : i \to j$
a~linear mapping $A_{\alpha}:
A_i \to A_j$. For example, the
problems of classifying
representations of the quivers
$$
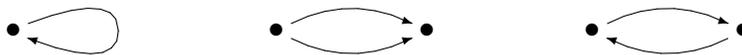
\special{em:linewidth 0.4pt}
\unitlength 1.00mm
\linethickness{0.4pt}
\begin{picture}(100.00,15.50)
\put(3.00,10.00){\makebox(0,0)[cc]{$\bullet$}}
\put(38.00,10.00){\makebox(0,0)[cc]{$\bullet$}}
\put(58.00,10.00){\makebox(0,0)[cc]{$\bullet$}}
\put(56.00,11.00){\vector(3,-1){0.2}}
\emline{40.00}{11.00}{1}{42.22}{11.96}{2}
\emline{42.22}{11.96}{3}{44.44}{12.60}{4}
\emline{44.44}{12.60}{5}{46.67}{12.94}{6}
\emline{46.67}{12.94}{7}{48.89}{12.98}{8}
\emline{48.89}{12.98}{9}{51.11}{12.70}{10}
\emline{51.11}{12.70}{11}{53.33}{12.11}{12}
\emline{53.33}{12.11}{13}{56.00}{11.00}{14}
\put(56.00,9.00){\vector(3,1){0.2}}
\emline{40.00}{9.00}{15}{42.22}{8.04}{16}
\emline{42.22}{8.04}{17}{44.44}{7.40}{18}
\emline{44.44}{7.40}{19}{46.67}{7.06}{20}
\emline{46.67}{7.06}{21}{48.89}{7.02}{22}
\emline{48.89}{7.02}{23}{51.11}{7.30}{24}
\emline{51.11}{7.30}{25}{53.33}{7.89}{26}
\emline{53.33}{7.89}{27}{56.00}{9.00}{28}
\emline{5.00}{11.00}{29}{8.10}{12.06}{30}
\emline{8.10}{12.06}{31}{10.74}{12.73}{32}
\emline{10.74}{12.73}{33}{12.92}{13.01}{34}
\emline{12.92}{13.01}{35}{14.63}{12.91}{36}
\emline{14.63}{12.91}{37}{15.88}{12.43}{38}
\emline{15.88}{12.43}{39}{16.67}{11.56}{40}
\emline{16.67}{11.56}{41}{17.00}{10.00}{42}
\put(5.00,9.00){\vector(-3,1){0.2}}
\emline{17.00}{10.00}{43}{16.77}{8.67}{44}
\emline{16.77}{8.67}{45}{16.07}{7.72}{46}
\emline{16.07}{7.72}{47}{14.92}{7.15}{48}
\emline{14.92}{7.15}{49}{13.30}{6.98}{50}
\emline{13.30}{6.98}{51}{11.21}{7.18}{52}
\emline{11.21}{7.18}{53}{8.67}{7.78}{54}
\emline{8.67}{7.78}{55}{5.00}{9.00}{56}
\put(80.00,10.00){\makebox(0,0)[cc]{$\bullet$}}
\put(100.00,10.00){\makebox(0,0)[cc]{$\bullet$}}
\put(98.00,11.00){\vector(3,-1){0.2}}
\emline{82.00}{11.00}{57}{84.22}{11.96}{58}
\emline{84.22}{11.96}{59}{86.44}{12.60}{60}
\emline{86.44}{12.60}{61}{88.67}{12.94}{62}
\emline{88.67}{12.94}{63}{90.89}{12.98}{64}
\emline{90.89}{12.98}{65}{93.11}{12.70}{66}
\emline{93.11}{12.70}{67}{95.33}{12.11}{68}
\emline{95.33}{12.11}{69}{98.00}{11.00}{70}
\put(82.00,9.00){\vector(-3,1){0.2}}
\emline{98.00}{9.00}{71}{95.78}{8.04}{72}
\emline{95.78}{8.04}{73}{93.56}{7.40}{74}
\emline{93.56}{7.40}{75}{91.33}{7.06}{76}
\emline{91.33}{7.06}{77}{89.11}{7.02}{78}
\emline{89.11}{7.02}{79}{86.89}{7.30}{80}
\emline{86.89}{7.30}{81}{84.67}{7.89}{82}
\emline{84.67}{7.89}{83}{82.00}{9.00}{84}
\end{picture}
$$\\[-10mm] are the problems of
classifying, respectively,
linear operators $A:U\to U$ (its
solution is the Jordan normal
form), pairs of linear mappings
$A: U\to V,\ B: U\to V$ (the
matrix pencil problem, solved by
Kronecker), and pairs of counter
linear mappings $A: U\to V,\ B:
V\to U$ (the contagredient
matrix pencil problem, solved in
\cite{pon} and studied in detail
in \cite{horn}).

Studying families of quiver
representations smoothly
depending on parameters, we can
independently reduce each
representation to canonical
form, but then we lose the
smoothness (and even the
continuity) relative to the
parameters. It leads to the
problem of reducing to normal
form by a smoothly depending on
parameters change of bases not
only the matrices of a given
representation, but of an
arbitrary family of
representations close to it.
This normal form is obtained
from the normal form of matrices
of the given representation by
adding to some of their entries
holomorphic functions of the
parameters that are zero for the
zero value of parameters. The
number of these entries must be
minimal to obtain the simplest
normal form.

This problem for representations
of the quiver
 \quiva
over $\mathbb C$ was solved by
Arnold \cite{arn} (see also
\cite[\S\,30]{arn3}). We solve
it for holomorphically depending
on parameters representations of
the quiver
\special{em:linewidth 0.4pt}
\unitlength 0.3mm
\linethickness{0.4pt}
\begin{picture}(17.00,15.50)(6,7)
\put(3.00,10.00){\makebox(0,0)[cc]{$\cdot$}}
\emline{5.00}{11.00}{701}{8.10}{12.06}{702}
\emline{8.10}{12.06}{703}{10.74}{12.73}{704}
\emline{10.74}{12.73}{705}{12.92}{13.01}{706}
\emline{12.92}{13.01}{707}{14.63}{12.91}{708}
\emline{14.63}{12.91}{709}{15.88}{12.43}{7010}
\emline{15.88}{12.43}{7011}{16.67}{11.56}{7012}
\emline{16.67}{11.56}{7013}{17.00}{10.00}{7014}
\put(5.00,9.00){\vector(-3,1){0.2}}
\emline{17.00}{10.00}{7015}{16.77}{8.67}{7016}
\emline{16.77}{8.67}{7017}{16.07}{7.72}{7018}
\emline{16.07}{7.72}{7019}{14.92}{7.15}{7020}
\emline{14.92}{7.15}{7021}{13.30}{6.98}{7022}
\emline{13.30}{6.98}{7023}{11.21}{7.18}{7024}
\emline{11.21}{7.18}{7025}{8.67}{7.78}{7026}
\emline{8.67}{7.78}{7027}{5.00}{9.00}{7028}
\end{picture}
over $\mathbb R$ and
representations of the quivers
\quivb and \quivc \!\!\! both
over $\mathbb C$ and over
$\mathbb R$. In the obtained
simplest normal forms, all the
summands to entries are
independent parameters. A normal
form with the minimal number of
independent parameters, but not
of the summands to entries, was
obtained in \cite{gal} (see also
\cite[\S\,30E]{arn3}) for
representations of the quiver
\special{em:linewidth 0.4pt}
\unitlength 0.3mm
\linethickness{0.4pt}
\begin{picture}(17.00,15.50)(6,7)
\put(3.00,10.00){\makebox(0,0)[cc]{$\cdot$}}
\emline{5.00}{11.00}{801}{8.10}{12.06}{802}
\emline{8.10}{12.06}{803}{10.74}{12.73}{804}
\emline{10.74}{12.73}{805}{12.92}{13.01}{806}
\emline{12.92}{13.01}{807}{14.63}{12.91}{808}
\emline{14.63}{12.91}{809}{15.88}{12.43}{8010}
\emline{15.88}{12.43}{8011}{16.67}{11.56}{8012}
\emline{16.67}{11.56}{8013}{17.00}{10.00}{8014}
\put(5.00,9.00){\vector(-3,1){0.2}}
\emline{17.00}{10.00}{8015}{16.77}{8.67}{8016}
\emline{16.77}{8.67}{8017}{16.07}{7.72}{8018}
\emline{16.07}{7.72}{8019}{14.92}{7.15}{8020}
\emline{14.92}{7.15}{8021}{13.30}{6.98}{8022}
\emline{13.30}{6.98}{8023}{11.21}{7.18}{8024}
\emline{11.21}{7.18}{8025}{8.67}{7.78}{8026}
\emline{8.67}{7.78}{8027}{5.00}{9.00}{8028}
\end{picture}
over $\mathbb R$ and  in
\cite{kag} (partial cases were
considered in
\cite{berg}--\cite{gar})
 for representations of the quiver
\special{em:linewidth 0.4pt}
\unitlength 0.3mm
\linethickness{0.4pt}
\begin{picture}(23.00,15.00)(6,7)
\put(3.00,10.00){\makebox(0,0)[cc]{$\cdot$}}
\put(23.00,10.00){\makebox(0,0)[cc]{$\cdot$}}
\put(21.00,11.00){\vector(3,-1){0.2}}
\emline{5.00}{11.00}{501}{7.22}{11.96}{502}
\emline{7.22}{11.96}{503}{9.44}{12.60}{504}
\emline{9.44}{12.60}{505}{11.67}{12.94}{506}
\emline{11.67}{12.94}{507}{13.89}{12.98}{508}
\emline{13.89}{12.98}{509}{16.11}{12.70}{5010}
\emline{16.11}{12.70}{5011}{18.33}{12.11}{5012}
\emline{18.33}{12.11}{5013}{21.00}{11.00}{5014}
\put(21.00,9.00){\vector(3,1){0.2}}
\emline{5.00}{9.00}{5015}{7.22}{8.04}{5016}
\emline{7.22}{8.04}{5017}{9.44}{7.40}{5018}
\emline{9.44}{7.40}{5019}{11.67}{7.06}{5020}
\emline{11.67}{7.06}{5021}{13.89}{7.02}{5022}
\emline{13.89}{7.02}{5023}{16.11}{7.30}{5024}
\emline{16.11}{7.30}{5025}{18.33}{7.89}{5026}
\emline{18.33}{7.89}{5027}{21.00}{9.00}{5028}
\end{picture}
over $\mathbb C$.


\section{Deformations of quiver representations} \label{s2}

Let $Q$ be a quiver with
vertices $1,\dots,t$. Its {\it
matrix representation} $A$ of
dimension $\vec
n=(n_1,\dots,n_t)\in
\{0,1,2,\dots\}^t$ over $\mathbb
F$ is given by assigning a
matrix $A_{\alpha}\in {\mathbb
F}^{\,n_j\times n_i}$ to each
arrow $\alpha: i\to j$. Denote
by ${\cal R}(\vec n,\mathbb F)$
the vector space of all matrix
representations of dimension
$\vec n$ over $\mathbb F$. An
{\it isomorphism} $S: A\to B$ of
$A,B\in{\cal R}(\vec n,\mathbb
F)$ is given by a sequence
$S=(S_1,\dots,S_t)$ of
non-singular matrices $S_i\in
{\rm Gl}(n_i,\mathbb F)$ such
that $B_{\alpha}=
S_jA_{\alpha}S_i^{-1}$ for each
arrow $\alpha: i\to j$.

By an {\it $\mathbb
F$-deformation} of $A\in{\cal
R}(\vec n,\mathbb F)$ is meant a
parametric matrix representation
${\cal
A}(\lambda_1,\dots,\lambda_k)$
(or for short ${\cal
A}(\vec\lambda)$, where
$\vec\lambda=(\lambda_1,\dots,\lambda_k)$),
whose entries are convergent in
a neighborhood of $\vec 0$ power
series of variables (they are
called {\it parameters})
$\lambda_1,\dots, \lambda_k$
over $\mathbb F$ such that
${\cal A}(\vec 0)=A$.

Two deformations ${\cal
A}(\vec\lambda)$ and ${\cal
B}(\vec\lambda)$ of $A\in{\cal
R}(\vec n,\mathbb F)$ are called
{\it equivalent} if there exists
a {\it deformation} ${\cal
I}(\vec\lambda)$ (its entries
are convergent in a neighborhood
of $\vec 0$ power series and
${\cal I}(\vec 0)=I$) of the
identity isomorphism
$I=(I_{n_1},\dots,I_{n_t}):A\to
A$  such that $$ {\cal
B}_{\alpha}(\vec\lambda)= {\cal
I}_j(\vec\lambda) {\cal
A}_{\alpha}(\vec\lambda) {\cal
I}_i^{-1}(\vec\lambda),\quad
\alpha:i\to j, $$ in a
neighborhood of $\vec 0$.

A deformation ${\cal
A}(\lambda_1,\dots,\lambda_k)$
of $A$ is called {\it versal} if
every deformation ${\cal
B}(\mu_1,\dots,\mu_l)$ of $A$ is
equivalent to a deformation
${\cal
A}(\varphi_1(\vec\mu),\dots,
\varphi_k(\vec\mu)),$ where
$\varphi_i(\vec\mu)$ are
convergent in a neighborhood of
$\vec 0$ power series such that
$\varphi_i(\vec 0)=\vec 0$. A
versal deformation ${\cal
A}(\lambda_1,\dots,\lambda_k)$
of $A$ is called {\it
miniversal} if there is no
versal deformation having less
than $k$ parameters.

For a matrix representation
$A\in{\cal R}(\vec n,\mathbb F)$
and a sequence
$C=(C_1,\dots,C_t)$, $C_i\in
{\mathbb F}^{\,n_i\times n_i}$,
we define the matrix
representation $[C,A]\in{\cal
R}(\vec n,\mathbb F)$ as
follows: $$
[C,A]_{\alpha}=C_jA_{\alpha}-A_{\alpha}C_i,
\quad \alpha:i\to j. $$

A miniversal deformation ${\cal
A}(\lambda_1,\dots, \lambda_k)$
of $A$ will be called {\it
simplest} if it is obtained from
$A$ by adding to certain $k$ of
its entries, respectively,
$\lambda_1$ to the first,
$\lambda_2$ to the
second$,\dots,$ and $\lambda_k$
to the $k$th. The next theorem
is a simple conclusion of a well
known fact.

\begin{theorem} \label{t2.1}
Let ${\cal
A}(\vec\lambda)=A+{\cal B}(\vec
\lambda)$,
$\vec\lambda=(\lambda_1,\dots,
\lambda_k)$, be an $\mathbb
F$-deformation of a matrix
representation $A\in{\cal
R}(\vec n,\mathbb F)$,
$\mathbb{F\in \{C,R\}}$, where
$k$ entries of ${\cal B}(\vec
\lambda)$ are the independent
parameters $\lambda_1,\dots,
\lambda_k$ and the other entries
are zeros. Then ${\cal A}(\vec
\lambda)$ is a simplest
miniversal deformation of $A$ if
and only if $$ {\cal R}(\vec
n,\mathbb F)={\cal P}_{\cal A}
\oplus {\cal T}_A, $$ where
${\cal P}_{\cal A}$ is the
$k$-dimensional vector space of
all ${\cal B}(\vec a)$, $\vec
a\in {\mathbb F}^{\,k}$, and
${\cal T}_A$ is the vector space
of all $[C,A]$, $C \in {\mathbb
F}^{\,n_1\times n_1}\times\dots
\times {\mathbb F}^{\,n_t\times
n_t}$.
\end{theorem}

\begin{proof}
Two subspaces of a vector space
$V$ are {\it transversal} if
their sum is equal to $V$. The
class of all isomorphic to
$A\in{\cal R}(\vec n,\mathbb F)$
matrix representations may be
considered as the orbit $A^G$ of
$A$ under the following action
of the group $G={\rm
GL}(n_1,{\mathbb
F})\times\dots\times {\rm
GL}(n_t,{\mathbb F})$ on the
space ${\cal R}(\vec n,\mathbb
F)$: $$ A_{\lambda}^S=S_j
A_{\lambda}S_i^{-1},\quad
\lambda: i\to j, $$ for all
$A\in{\cal R}(\vec n,\mathbb
F)$, $S=(S_1,\dots,S_t) \in G$,
and arrows $\lambda$. A
deformation ${\cal
A}(\vec\lambda)$ of a matrix
representation $A\in{\cal
R}(\vec n,\mathbb F)$ is called
a {\it transversal to the orbit
$A^G$ at the point $A$} if the
space ${\cal R}(\vec n,\mathbb
F)$ is the sum of the space
${\cal A}_*{\mathbb F}^{\,k}$
(that is, of the image of the
linearization ${\cal A}_*$ of
$\cal A(\vec\lambda)$ near $A$;
the linearization means that
only first derivatives matter)
and of the tangent space to the
orbit $A^G$ at the point $A$.
The following fact is well known
(see, for example, \cite[Section
1.6]{arn2} and \cite{arn}): {\it
a transversal (of the minimal
dimension) to the orbit is a
(mini)versal deformation.}

It proves the theorem since
${\cal P}_{\cal A}$ is the space
${\cal A}_*{\mathbb F}^{\,k}$
and ${\cal T}_A$ is the tangent
space to the orbit $A^G$ at the
point $A$; the last follows from
\begin{multline*}
A_{\lambda}^{I+\varepsilon C}=
(I+\varepsilon
C_j)A_{\lambda}(I+\varepsilon
C_i)^{-1}= (I+\varepsilon
C_j)A_{\lambda}(I-\varepsilon
C_i+\varepsilon^2 C_i-\cdots) \\
=A_{\lambda}+ \varepsilon(C_j
A_{\lambda}- A_{\lambda}C_i)
+\varepsilon^2...\, ,
\end{multline*}
for all $C=(C_1,\dots,C_t)$,
$C_i\in {\mathbb F}^{\,n_i\times
n_i}$, small $\varepsilon$, and
arrows $\lambda: i\to j$.
\end{proof}

\begin{corollary} \label{c2.1}
There exists a simplest
miniversal $\mathbb
F$-deformation for every matrix
representation over
$\mathbb{F\in\{C,R\}}$.
\end{corollary}

\begin{proof}
Let $A\in{\cal R}(\vec n,\mathbb
F)$, let $T_1,\dots,T_r$ be a
basis of the space ${\cal T}_A$,
and let $E_1,\dots,E_l$ be the
basis of ${\cal R}(\vec
n,\mathbb F)$ consisting of all
matrix representations of
dimension $\vec n$ such that
each of theirs has one entry
equaling 1 and the others
equaling 0. Removing from the
sequence $T_1,\dots,
T_r,E_1,\dots,E_l$ every
representation that is a linear
combination of the preceding
representations, we obtain a new
basis $T_1,\dots, T_r,
E_{i_1},\dots,E_{i_k}$ of the
space $ {\cal R}(\vec n,\mathbb
F)$. By Theorem \ref{t2.1}, the
deformation $$ {\cal
A}(\lambda_1,\dots, \lambda_k)=
A+\lambda_1
E_{i_1}+\dots+\lambda_kE_{i_k}
$$ is a simplest miniversal
deformation of $A$ since $
E_{i_1},\dots, E_{i_k}$ is a
basis of ${\cal P}_{\cal A}$ and
${\cal R}(\vec n,\mathbb
F)={\cal P}_{\cal A} \oplus
{\cal T}_A$.
\end{proof}

By a set of canonical
representations of a quiver $Q$,
we mean an arbitrary set of
``nice'' matrix representations
such that every class of
isomorphic representations
contains exactly one
representation from it. Clearly,
it suffices to study
deformations of the canonical
representations.

Arnold \cite{arn} obtained a
simplest miniversal deformation
of the Jordan matrices (i.e.,
canonical representations of the
quiver \quiva\!\!\!).
 In the remaining of the article, we obtain simplest miniversal deformations of canonical representations of the quiver
\special{em:linewidth 0.4pt}
\unitlength 0.3mm
\linethickness{0.4pt}
\begin{picture}(17.00,15.50)(6,7)
\put(3.00,10.00){\makebox(0,0)[cc]{$\cdot$}}
\emline{5.00}{11.00}{801}{8.10}{12.06}{802}
\emline{8.10}{12.06}{803}{10.74}{12.73}{804}
\emline{10.74}{12.73}{805}{12.92}{13.01}{806}
\emline{12.92}{13.01}{807}{14.63}{12.91}{808}
\emline{14.63}{12.91}{809}{15.88}{12.43}{8010}
\emline{15.88}{12.43}{8011}{16.67}{11.56}{8012}
\emline{16.67}{11.56}{8013}{17.00}{10.00}{8014}
\put(5.00,9.00){\vector(-3,1){0.2}}
\emline{17.00}{10.00}{8015}{16.77}{8.67}{8016}
\emline{16.77}{8.67}{8017}{16.07}{7.72}{8018}
\emline{16.07}{7.72}{8019}{14.92}{7.15}{8020}
\emline{14.92}{7.15}{8021}{13.30}{6.98}{8022}
\emline{13.30}{6.98}{8023}{11.21}{7.18}{8024}
\emline{11.21}{7.18}{8025}{8.67}{7.78}{8026}
\emline{8.67}{7.78}{8027}{5.00}{9.00}{8028}
\end{picture}
over $\mathbb R$ and of the
quivers \!\quivb\!\!\! and
\mbox{\quivc\!\!\!} both over
$\mathbb C$ and over $\mathbb
R$.

\begin{remark}
Arnold \cite{arn} proposed an
easy method to obtain a
miniversal (but not a simplest
miniversal) deformation of a
matrix under similarity by
solving a certain system of
linear equations. The method is
of considerable current use (see
\cite{kag, berg, gar, gar1}).
Although we do not use it in the
next sections, now we show how
to extend this method to quiver
representations.

The space ${\cal R}(\vec
n,\mathbb F)$ may be considered
as a Euclidean space with scalar
product $$ \langle A,B \rangle =
\sum_{\alpha\in Q_1}{\rm
tr}(A_{\alpha}B_{\alpha}^*), $$
where $Q_1$ is the set of arrows
of $Q$ and $B_{\alpha}^*$ is the
adjoint of $B_{\alpha}$.

Let $A\in{\cal R}(\vec n,\mathbb
F)$ and let $T_1,\dots,T_k$ be a
basis of the orthogonal
complement ${\cal T}_A^{\bot}$
to the tangent space ${\cal
T}_A$. The deformation
\begin{equation}       \label{2.2}
{\cal A}(\lambda_1,\dots,
\lambda_k)=
A+\lambda_1T_1+\dots+\lambda_kT_k
\end{equation}
is a miniversal deformation
(since it is a transversal of
the minimal dimension to the
orbit of $A$) called an {\it
orthogonal miniversal
deformation}.

For every arrow $\alpha:i\to j$,
we denote $b(\alpha):= i$ and
$e(\alpha):= j$. By the proof of
Theorem \ref{t2.1}, $B\in{\cal
T}_A^{\bot}$ if and only if
$\langle B, [C,A]\rangle=0$ for
all $C\in {\mathbb
F}^{\,n_1\times n_1}\times\dots
\times {\mathbb F}^{\,n_t\times
n_t}$. Then
\begin{multline*}
\langle B, [C,A]\rangle=
\sum_{\alpha\in Q_1}{\rm tr}
(B_{\alpha}
(C_{e(\alpha)}A_{\alpha}-
A_{\alpha} C_{b(\alpha)})^*)\\
=\sum_{\alpha\in Q_1}{\rm tr}
(B_{\alpha}
A_{\alpha}^*C_{e(\alpha)}^* -
B_{\alpha} C_{b(\alpha)}^*
A_{\alpha}^*) =\sum_{i=1}^t{\rm
tr}(S_iC_i^*)=0,
\end{multline*}
where $$
S_i:=\sum_{e(\alpha)=i}B_{\alpha}A_{\alpha}^*-
\sum_{b(\alpha)=i} A_{\alpha}^*
B_{\alpha}. $$ Taking $C_i=S_i$
for all vertices $i=1,\dots,t$,
we obtain $S_i=0$.

Therefore, {\it every orthogonal
miniversal deformation of $A$
has the form \eqref{2.2}, where
$T_1,\dots,T_k$ is a fundamental
system of solutions of the
system of homogeneous matrix
equations $$
\sum_{e(\alpha)=i}X_{\alpha}A_{\alpha}^*=
\sum_{b(\alpha)=i} A_{\alpha}^*
X_{\alpha}, \quad i=1,\dots,t,
$$ with unknowns
$T=\{X_{\alpha}\,|\,\alpha\in
Q_1\}$.}
\end{remark}


\section{Deformations of matrices} \label{s5}

In this section, we obtain a
simplest miniversal $\mathbb
R$-deformation of a real matrix
under similarity.

Let us denote
\begin{equation}       \label{5.2''}
J_r^{\mathbb C}(\lambda)=
J_r(\lambda):=\begin{bmatrix}
\lambda&1&&\\&\lambda&\ddots&\\&&\ddots&1
\\ &&&\lambda
\end{bmatrix}, \quad J_r:=J_r(0);
\end{equation}
and, for $\lambda=a+bi\in
\mathbb C$ $(b\ge 0)$, denote
$J_r^{\mathbb R}(\lambda):=
J_r(\lambda)$ if $b=0$ and
\begin{equation}       \label{5.2'}
J_r^{\mathbb R}(\lambda):=
\begin{bmatrix}
T_{ab}&I_2&&\\&T_{ab}&\ddots&\\&&\ddots&I_2\\&&&T_{ab}
\end{bmatrix}\ {\rm if}\ b>0,\ {\rm where}\ T_{ab}:=
\begin{bmatrix}
a&b\\-b&a
\end{bmatrix},
\end{equation}
(the size of $J_r(\lambda),\
J_r^{\mathbb C}(\lambda)$ and $
J_r^{\mathbb R}(\lambda)$ is
$r\times r$).

Clearly, every square matrix
over $\mathbb F\in\{ \mathbb C,
\mathbb R\}$ is similar to a
matrix of the form
\begin{equation}       \label{5.3''}
\oplus_i\Phi^{\mathbb
F}({\lambda_i}),\quad
\lambda_i\ne\lambda_j\ {\rm if}\
i\ne j,
\end{equation}
uniquely determined up to
permutations of summands, where
\begin{equation}       \label{5.3'''}
\Phi^{\mathbb
F}({\lambda_i}):={\rm
diag}(J^{\mathbb
F}_{s_{i1}}({\lambda_i}),\,
J^{\mathbb
F}_{s_{i2}}({\lambda_i}),\dots),\quad
s_{i1}\ge s_{i2}\ge\cdots.
\end{equation}

 Let
\begin{equation}       \label{5.0}
{\cal H}=[H_{ij}]
\end{equation}
 be a parametric block
matrix with $p_i\times q_j$
blocks $H_{ij}$ of the form

\begin{equation}       \label{5.1}
H_{ij}=\left[
\begin{tabular}{cc}
             $*$& \\[-2.5mm]
             $\vdots$&\!\!\!\Large 0\\[-2mm]
             $*$&
\end{tabular}
\right] \  {\rm if}\ p_i\le q_j,
\quad
H_{ij}=\left[\begin{tabular}{c}
                       \Large 0 \\[-1mm]
                                          $\!\!* \cdots *\!\!$
 \end{tabular}\right]
\  {\rm if}\ p_i>q_j,
\end{equation}
where the stars denote
independent parameters.

Arnold \cite{arn} (see also
\cite[\S\,30]{arn3}) proved that
one of the simplest miniversal
$\mathbb C$-deformations of the
matrix \eqref{5.3''} for
$\mathbb {F=C}$ is
$\oplus_i(\Phi^{\mathbb
C}({\lambda_i})+{\cal H}_i)$,
where ${\cal H}_i$ is of the
form \eqref{5.0}. Galin
\cite{gal} (see also
\cite[\S\,30E]{arn3}) showed
that one of the miniversal
$\mathbb R$-deformations of the
matrix \eqref{5.3''} for
$\mathbb {F=R}$ is
$\oplus_i(\Phi^{\mathbb
R}({\lambda_i})+{\cal
H}_{\lambda_i})$, where ${\cal
H}_{\lambda }$ (${\lambda }\in
\mathbb R)$ is of the form
\eqref{5.0} and ${\cal
H}_{\lambda }$ (${\lambda
}\notin \mathbb R)$ is obtained
from a matrix of the form
\eqref{5.0} by the replacement
of its entries $\alpha + \beta
i$ with $2\times 2$ blocks
$T_{\alpha \beta }$ (see
\eqref{5.2'}). For example, a
real $4\times 4$ matrix with two
Jordan $2\times 2$ blocks with
eigenvalues $x\pm iy\ (y\ne 0)$
has a miniversal $\mathbb
R$-deformation \begin{equation}
\label{5.2}
\begin{bmatrix}
x&y&1&0\\-y&x&0&1\\0&0&x&y\\0&0&-y&x
\end{bmatrix}+\begin{bmatrix}
\alpha_1&\beta_1&0&0\\
 -\beta_1&\alpha_1&0&0\\
\alpha_2&\beta_2&0&0\\ -\beta_2&
\alpha_2&0&0\\
\end{bmatrix}
\end{equation}
with the parameters $
\alpha_1,\beta_1,\alpha_2,\beta_2$.
We prove that a simplest
miniversal $\mathbb
R$-deformation of this matrix
may be obtained by the
replacement of the second column
$(\beta_1,\alpha_1,\beta_2,\alpha_2)^T$
in \eqref{5.2} with
$(0,0,0,0)^T$.

\begin{theorem}[Arnold \cite{arn} for $\mathbb F=\mathbb C$]   \label{t5.1}
One of the simplest miniversal
$\mathbb F$-deforma\-tions of
the canonical matrix
\eqref{5.3''} under similarity
over ${\mathbb F}\in\{ \mathbb
C, \mathbb R\}$ is
$\oplus_i(\Phi^{\mathbb
F}({\lambda_i})+{\cal H}_i)$,
where ${\cal H}_i$ is of the
form \eqref{5.0}.
\end{theorem}

 \begin{proof}
Let $A$ be the matrix
\eqref{5.3''}. By Theorem
\ref{t2.1}, we must prove that
for every $M\in {\mathbb
F}^{\,m\times m}$ there exists
$S\in {\mathbb F}^{\,m\times m}$
such that
\begin{equation}       \label{5.3}
M+SA-AS=N,
\end{equation}
 where $N$ is obtained from $\oplus_i{\cal H}_i$ by replacing its stars with elements of ${\mathbb F}$  and is uniquely determined by $M$.
The matrix $A$ is block-diagonal
with diagonal blocks of the form
$J_r^{\mathbb F}(\lambda)$. We
apply the same partition into
blocks to $M$ and $N$ and
rewrite the equality \eqref{5.3}
for blocks:
\begin{equation*}
M_{ij}+S_{ij}A_j-A_iS_{ij}=N_{ij}.
\end{equation*}
The theorem follows from the
next lemma.
\end{proof}

\begin{lemma}  \label{l5.1}
 For given $J_p^{\mathbb F}(\lambda)$, $J_q^{\mathbb F}(\mu)$, and for every matrix $M\in {\mathbb F}^{\,p\times q}$ there exists a
matrix $S\in {\mathbb
F}^{\,p\times q}$ such that
$M+SJ_q^{\mathbb F}(\mu)-
J_p^{\mathbb F}(\lambda)S=0$ if
$\lambda\ne \mu$, and
$M+SJ_q^{\mathbb F}(\mu)-
J_p^{\mathbb F}(\lambda)S=H$ if
$\lambda= \mu$, where $H$ is of
the form \eqref{5.1} with
elements from ${\mathbb F}$
instead of the stars; moreover,
$H$ is uniquely determined by
$M$.
\end{lemma}

\begin{proof}
If $\lambda\ne\mu$ then
$J_q^{\mathbb F}(\mu)$ and
$J_p^{\mathbb F}(\lambda)$ have
no common eigenvalues, the matrix
$S$ exists by \cite[Sect.
8]{gan}.

Let $\lambda=\mu$ and let
${\mathbb F}={\mathbb C}$ or
$\lambda\in{\mathbb R}$.
 Put $C:=SJ_q^{\mathbb F}(\lambda)- J_p^{\mathbb F}(\lambda)S= SJ_q-J_pS$.
As is easily seen, $C$ is an
arbitrary matrix $[c_{ij}]$ (for
a suitable $S$) satisfying the
condition: if its diagonal
$C_t=\{c_{ij}\,|\,i-j=t\}$
contains both an entry from the
first column and an entry from
the last row, then the sum of
entries of this diagonal is
equal to zero. It proves the
lemma in this case.

Let $\lambda=\mu$, ${\mathbb
F}={\mathbb R}$ and
$\lambda=a+bi$, $b>0$. Then
$p=2m$ and $q=2n$ for certain
$m$ and $n$. We must prove that
every $2m\times 2n$ matrix $M$
can be reduced to a uniquely
determined matrix $H$ of the
form (\ref{5.1}) (with real
numbers instead of the stars) by
transformations
\begin{equation}       \label{5.3'}
M \longmapsto M+SJ_{2n}^{\mathbb
R}(\lambda)- J_{2m}^{\mathbb
R}(\lambda)S,\quad S\in{\mathbb
F}^{\,2m\times 2n}.
\end{equation}
Let us partition $M$ and $S$
into $2\times 2$ blocks $M_{ij}$
and $S_{ij}$, where $1\le i\le
m$ and $1\le j\le n.$ For every
$2\times 2$ matrix $P=[p_{ij}]$,
define (see \eqref{5.2'}) $$
P':=P T_{01}- T_{01}P=
\begin{bmatrix}
-p_{12}-p_{21}& p_{11}-p_{22}\\
p_{11}-p_{22}& p_{12}+p_{21}
\end{bmatrix}.
$$

By \eqref{5.2'}, the
transformation \eqref{5.3'} has
the form $M \mapsto
M+S(T_{ab}\oplus\dots\oplus
T_{ab})-
(T_{ab}\oplus\dots\oplus
T_{ab})S+ SJ_{2n}^2- J_{2m}^2S=
M+b[S(T_{01}\oplus\dots\oplus
T_{01})-
(T_{01}\oplus\dots\oplus
T_{01})S]+ SJ_{2n}^2-
J_{2m}^2S$, that is
\begin{multline} \label{5.4}
M \longmapsto
\begin{bmatrix}
M_{11}& M_{12}&\cdots&M_{1n}\\
M_{21}& M_{22}&\cdots&M_{2n}\\
\hdotsfor{4}\\ M_{m1}&
M_{m2}&\cdots&M_{mn}
\end{bmatrix}
 +b\begin{bmatrix}
S'_{11}&
S'_{21}&\cdots&S'_{1n}\\
S'_{21}&
S'_{22}&\cdots&S'_{2n}\\
\hdotsfor{4}\\ S'_{m1}&
S'_{m2}&\cdots&S'_{mn}
\end{bmatrix}\\
+\begin{bmatrix}
0&S_{11}&\cdots&S_{1,n-1}\\
0&S_{21}&\cdots&S_{2,n-1}\\
\hdotsfor{4}\\
0&S_{m1}&\cdots&S_{m,n-1}
\end{bmatrix}-
\begin{bmatrix}
S_{21}& S_{22}&\cdots&S_{2n}\\
\hdotsfor{4}\\ S_{m1}&
S_{m2}&\cdots&S_{mn}\\
0&0&\cdots&0
\end{bmatrix}.
\end{multline}

Let first $m\le n$. If $m>1$, we
make $M_{mn}=0$ selecting
$S_{mn}'$ and $S_{m,n-1}$. To
preserve it, we must further
take the transformations
\eqref{5.4} with $S$ satisfying
$bS'_{mn}+S_{m,n-1}=0$; that is,
$S'_{mn}=-b^{-1}S_{m,n-1}$ and
$S_{m,n-1}=\matr{-\alpha}{\beta}{\beta}{\alpha}$
with arbitrary $\alpha$ and
$\beta$.

Selecting
$S'_{m,n-1}=\matr{-2\beta}{-2\alpha}{-2\alpha}{2\beta}$
and $S_{m,n-2}$, we make
$M_{m,n-1}=0$. To preserve it,
we must take
$bS'_{m,n-1}+S_{m,n-2}=0$; that
is,
$S'_{m,n-1}=-b^{-1}S_{m,n-2}$
and
$S_{m,n-2}=\matr{-\alpha}{\beta}{\beta}{\alpha}$
with arbitrary $\alpha$ and
$\beta$; and so on until obtain
$M_{m2}=\dots=M_{mn}=0$. To
preserve theirs, we must take
$S_{m1}=\matr{-\alpha}{\beta}{\beta}{\alpha}$
with arbitrary $\alpha$ and
$\beta$ and suitable
$S_{m2},\dots,S_{mn}$. Then
$M_{m1} \mapsto M_{m1}+b
\matr{-2\beta}{-2\alpha}{-2\alpha}{2\beta},$
we make $M_{m1}=
\matr{\gamma}{0}{\delta}{0}$,
where $\gamma$ and $\delta$ are
uniquely determined.

We have reduced the last strip
of $M$ to the form
\begin{equation}       \label{5.5}
[M_{m1}\cdots M_{mn}]=
\begin{bmatrix}
\gamma&0&\cdots&0\\
\delta&0&\cdots&0
\end{bmatrix}.
\end{equation}
To preserve it, we must take
$S_{m1}=\dots= S_{m,n-1}=
S_{mn}'=0$ since the number of
zeros in $M_{m1},\dots,M_{mn}$
is equal to the number of
parameters in $S_{m1},\dots,
S_{m,n-1}, S_{mn}'$.

The next to last strip of $M$
transforms as follows:
$[M_{m-1,1}\cdots
M_{m-1,n}]\mapsto [M_{m-1,1}$
$\cdots M_{m-1,n}]+ b[S'_{
m-1,1}\cdots S'_{ m-1,n}]+
[0\,S_{ m-1,1}\cdots S_{
m-1,n-1}]-[0\cdots 0\,S_{mn}]$.
In the same way, we reduce it to
the form $$ [M_{m-1,1}\cdots
M_{m-1,n}]=
\begin{bmatrix}
\tau&0&\cdots&0\\ \nu&0&\cdots&0
\end{bmatrix}
$$ taking, say, $S_{mn}=0$. We
must prove that $\tau$ and $\nu$
are uniquely determined for all
$S_{mn}$ such that $S'_{mn}=0$.
It may be proved as for the
$\gamma$ and $\delta$ from
\eqref{5.5} since the next to
last horizontal strip of $M$,
without the last block, is
transformed as the last strip:
$
[M_{m-1,1}\cdots M_{m-1,n-1}]
\mapsto [M_{m-1,1}\cdots
M_{m-1,n-1}]+ b[S'_{m-1,1}\cdots
S'_{m-1,n-1}]+
[0\,S_{m-1,1}\cdots
S_{m-1,n-2}]$ (recall that $m\le
n$, so this equality is not
empty for $m>1$).

We repeat this procedure until
reduce $M$ to the form
\eqref{5.1}.

If $m>n$, we reduce $M$ to the
form \eqref{5.1} starting with
the first vertical strip.
\end{proof}


\section{Deformations of matrix pencils} \label{s3}

The canonical form problem for
pairs of matrices
$A,B\in{\mathbb F}^{\,m\times
n}$ under transformations of
simultaneous equivalence $$
(A,B)\mapsto
(SAR^{-1},SBR^{-1}),\quad
S\in{\rm GL}(m,\mathbb F),\ \
R\in{\rm GL}(n,\mathbb F), $$
(that is, for representations of
the quiver \quivb\!\!\!) was
solved by Kronecker: each pair
is uniquely, up to permutation
of summands, reduced to a direct
sum of pairs of the form (see
\eqref{5.2''}--\eqref{5.2'})
\begin{equation}       \label{3.1}
(I,J_r^{\mathbb F}(\lambda)),\
(J_r,I),\ (F_r,K_r),\
(F_r^T,K_r^T),
\end{equation}
where $\lambda=a+bi\in {\mathbb
C}\ (b\ge 0$ if ${\mathbb
F}={\mathbb R}$) and
\begin{equation}       \label{3.1o}
F_r=\begin{bmatrix}
         1&&0\\
         0&\ddots&\\
         &\ddots&1\\
         0&&0
         \end{bmatrix},\quad
K_r=\begin{bmatrix}
         0&&0\\
         1&\ddots&\\
         &\ddots&0\\
         0&&1
         \end{bmatrix}
\end{equation}
are matrices of size  $r\times
(r-1),\ r\times (r-1)$, $r\ge
1$.

A miniversal, but not a simplest
miniversal, deformation of the
canonical pairs of matrices
under simultaneous similarity
was obtained in \cite{kag},
partial cases were considered in
\cite{berg}--\cite{gar}.

     Denote by $ 0^{\uparrow}$ (resp., $0^{\downarrow},\ 0^{\leftarrow},\, 0^{\rightarrow}$) a matrix, in which all entries are zero except for the entries of the first row (resp., the last row, the first column, the last column) that are independent parameters; and denote by $Z$ the $p\times q$ matrix, in which the first $\max\{q-p,0\}$ entries of the first row are independent parameters and the other entries are zeros:
\begin{equation}       \label{3.1a}
0^{\uparrow}=\begin{bmatrix}
         *&\cdots&*\\
         0&\cdots&0\\
         \hdotsfor{3}\\
          0&\cdots&0
         \end{bmatrix},\ \
Z=\left[\begin{tabular}{cccccc}
         $*$ & $\cdots$ & $*$ & 0 & $\cdots$ & 0\\
             &          &     &   & $\ddots$ &  \\
             & \LARGE 0 &     & 0 & $\cdots$ & 0
         \end{tabular}\right].
\end{equation}

\begin{theorem}     \label{t3.1}
Let
\begin{equation}       \label{3.1b}
(A,B)=\bigoplus_{i=1}^l(F_{p_i},
K_{p_i})\oplus (I,C) \oplus
(D,I)\oplus
\bigoplus_{i=1}^r(F_{q_i}^T,
K_{q_i}^T)
\end{equation}
be a canonical pair of matrices
under simultaneous equivalence
over ${\mathbb F}\in \{\mathbb
C, \mathbb R\}$, where $C$ is of
the form \eqref{5.3''},
$D=\Phi^{\mathbb F}(0)$ $($see
\eqref{5.3'''}$)$, and\footnote{We
use a special ordering of
summands in the decomposition
\eqref{3.1b} to obtain $\cal A$
and ${\cal B}$ in the upper
block triangular form except for
blocks in $\tilde C$ and $\tilde
D$.} $p_1\le\dots\le p_l,$
$q_1\ge\dots\ge q_r.$ Then one
of the simplest  miniversal
${\mathbb F}$-deformations of
$(A,B)$ has the form $({\cal
A},{\cal B})=$ $$ {\rm \left(
\special{em:linewidth 0.4pt}
\unitlength 0.60mm
\linethickness{0.4pt}
\begin{picture}(212.00,60.00)(29,50)
\emline{130.00}{60.00}{1}{30.00}{60.00}{2}
\emline{90.00}{100.00}{3}{90.00}{0.00}{4}
\emline{80.00}{100.00}{5}{80.00}{40.00}{6}
\emline{80.00}{40.00}{7}{130.00}{40.00}{8}
\emline{130.00}{50.00}{9}{70.00}{50.00}{10}
\emline{70.00}{50.00}{11}{70.00}{100.00}{12}
\put(35.00,95.00){\makebox(0,0)[cc]{$F_{p_1}$}}
\put(45.00,85.00){\makebox(0,0)[cc]{$F_{p_2}$}}
\put(55.00,75.00){\makebox(0,0)[cc]{$\ddots$}}
\put(65.00,65.00){\makebox(0,0)[cc]{$F_{p_l}$}}
\put(75.00,80.00){\makebox(0,0)[cc]{\rm\Large
0}}
\put(85.00,55.00){\makebox(0,0)[cc]{\rm\Large
0}}
\put(110.00,55.00){\makebox(0,0)[cc]{\rm\Large
0}}
\put(60.00,30.00){\makebox(0,0)[cc]{\rm\Large
0}}
\put(75.00,55.00){\makebox(0,0)[cc]{$I$}}
\put(85.00,45.00){\makebox(0,0)[cc]{$\tilde
D$}}
\put(60.00,90.00){\makebox(0,0)[cc]{\rm\Large
0}}
\put(40.00,70.00){\makebox(0,0)[cc]{\rm\Large
0}}
\put(95.00,35.00){\makebox(0,0)[cc]{$F_{q_1}^T$}}
\put(105.00,25.00){\makebox(0,0)[cc]{$F_{q_2}^T$}}
\put(115.00,15.00){\makebox(0,0)[cc]{$\ddots$}}
\put(125.00,5.00){\makebox(0,0)[cc]{$F_{q_r}^T$}}
\put(120.00,30.00){\makebox(0,0)[cc]{\rm\Large
0}}
\put(100.00,10.00){\makebox(0,0)[cc]{\rm\Large
0}}
\put(85.00,95.00){\makebox(0,0)[cc]{$0^{\downarrow}$}}
\put(85.00,85.00){\makebox(0,0)[cc]{$0^{\downarrow}$}}
\put(85.00,75.00){\makebox(0,0)[cc]{$\vdots$}}
\put(85.00,65.00){\makebox(0,0)[cc]{$0^{\downarrow}$}}
\put(95.00,80.00){\makebox(0,0)[cc]{$0^{\rightarrow}$}}
\put(105.00,80.00){\makebox(0,0)[cc]{$0^{\rightarrow}$}}
\put(115.00,80.00){\makebox(0,0)[cc]{$\cdots$}}
\put(125.00,80.00){\makebox(0,0)[cc]{$0^{\rightarrow}$}}
\put(125.00,45.00){\makebox(0,0)[cc]{$0^{\rightarrow}$}}
\put(115.00,45.00){\makebox(0,0)[cc]{$\cdots$}}
\put(105.00,45.00){\makebox(0,0)[cc]{$0^{\rightarrow}$}}
\put(95.00,45.00){\makebox(0,0)[cc]{$0^{\rightarrow}$}}
\put(135,50){\makebox(0,0)[cc]{,}}
\put(145.00,95.00){\makebox(0,0)[cc]{$K_{p_1}$}}
\put(155.00,85.00){\makebox(0,0)[cc]{$K_{p_2}$}}
\put(165.00,75.00){\makebox(0,0)[cc]{$\ddots$}}
\put(175.00,65.00){\makebox(0,0)[cc]{$K_{p_l}$}}
\put(195.00,55.00){\makebox(0,0)[cc]{\rm\Large
0}}
\put(170.00,30.00){\makebox(0,0)[cc]{\rm\Large
0}}
\put(185.00,55.00){\makebox(0,0)[cc]{$\tilde
C$}}
\put(195.00,45.00){\makebox(0,0)[cc]{$I$}}
\put(150.00,70.00){\makebox(0,0)[cc]{\rm\Large
0}}
\put(205.00,35.00){\makebox(0,0)[cc]{$K_{q_1}^T$}}
\put(215.00,25.00){\makebox(0,0)[cc]{$K_{q_2}^T$}}
\put(225.00,15.00){\makebox(0,0)[cc]{$\ddots$}}
\put(235.00,5.00){\makebox(0,0)[cc]{$K_{q_r}^T$}}
\put(210.00,10.00){\makebox(0,0)[cc]{\rm\Large
0}}
\put(155.00,95.00){\makebox(0,0)[cc]{$Z$}}
\put(165.00,95.00){\makebox(0,0)[cc]{$\cdots$}}
\put(175.00,95.00){\makebox(0,0)[cc]{$Z$}}
\put(175.00,85.00){\makebox(0,0)[cc]{$\vdots$}}
\put(175.00,75.00){\makebox(0,0)[cc]{$Z$}}
\put(165.00,85.00){\makebox(0,0)[cc]{$\ddots$}}
\put(215.00,35.00){\makebox(0,0)[cc]{$Z^T$}}
\put(225.00,35.00){\makebox(0,0)[cc]{$\cdots$}}
\put(235.00,35.00){\makebox(0,0)[cc]{$Z^T$}}
\put(235.00,25.00){\makebox(0,0)[cc]{$\vdots$}}
\put(235.00,15.00){\makebox(0,0)[cc]{$Z^T$}}
\put(225.00,25.00){\makebox(0,0)[cc]{$\ddots$}}
\put(185.00,95.00){\makebox(0,0)[cc]{$0^{\uparrow}$}}
\put(185.00,85.00){\makebox(0,0)[cc]{$0^{\uparrow}$}}
\put(185.00,75.00){\makebox(0,0)[cc]{$\vdots$}}
\put(185.00,65.00){\makebox(0,0)[cc]{$0^{\uparrow}$}}
\put(220.00,95.00){\makebox(0,0)[cc]{$0^{\uparrow}$}}
\put(220.00,85.00){\makebox(0,0)[cc]{$0^{\uparrow}$}}
\put(220.00,75.00){\makebox(0,0)[cc]{$\vdots$}}
\put(220.00,65.00){\makebox(0,0)[cc]{$0^{\uparrow}$}}
\put(205.00,55.00){\makebox(0,0)[cc]{$0^{\leftarrow}$}}
\put(215.00,55.00){\makebox(0,0)[cc]{$0^{\leftarrow}$}}
\put(225.00,55.00){\makebox(0,0)[cc]{$\cdots$}}
\put(235.00,55.00){\makebox(0,0)[cc]{$0^{\leftarrow}$}}
\put(195.00,80.00){\makebox(0,0)[cc]{\rm\Large
0}}
\put(220.00,45.00){\makebox(0,0)[cc]{\rm\Large
0}}
\emline{145.00}{0.00}{13}{140.00}{0.00}{14}
\emline{140.00}{0.00}{15}{140.00}{100.00}{16}
\emline{140.00}{100.00}{17}{145.00}{100.00}{18}
\emline{235.00}{100.00}{19}{240.00}{100.00}{20}
\emline{240.00}{100.00}{21}{240.00}{0.00}{22}
\emline{240.00}{0.00}{23}{235.00}{0.00}{24}
\emline{240.00}{60.00}{25}{140.00}{60.00}{26}
\emline{200.00}{100.00}{27}{200.00}{0.00}{28}
\emline{240.00}{40.00}{29}{190.00}{40.00}{30}
\emline{190.00}{40.00}{31}{190.00}{100.00}{32}
\emline{180.00}{100.00}{33}{180.00}{50.00}{34}
\emline{180.00}{50.00}{35}{240.00}{50.00}{36}
\emline{125.00}{0.00}{37}{130.00}{0.00}{38}
\emline{130.00}{0.00}{39}{130.00}{100.00}{40}
\emline{130.00}{100.00}{41}{125.00}{100.00}{42}
\emline{35.00}{100.00}{43}{30.00}{100.00}{44}
\emline{30.00}{100.00}{45}{30.00}{0.00}{46}
\emline{30.00}{0.00}{47}{35.00}{0.00}{48}
\end{picture}
\right),} $$

\noindent where $\tilde C$ and
$\tilde D$ are simplest
miniversal $\mathbb
F$-deformations of $C$ and $D$
under similarity (for instance,
given by Theorem \ref{t5.1}).
\end{theorem}

Let us denote by $S^{\succ}$
(resp., $S^{\prec},\
S^{\curlyvee},\
S^{\curlywedge}$) the matrix
that is obtained from a matrix
$S$ by removing of its first
column (resp., last column,
first row, last row), and denote
by $S_{\rhd}$ (resp.,
$S_{\lhd},$ $S_{\triangledown}$,
$S_{\vartriangle}$) the matrix
that is obtained from a matrix
$S$ by connecting of the zero
column to the right (resp.,
zero column to the left, zero
row at the bottom, zero row at
the top).

The following equalities hold
for every $p\times q$ matrix
$S$:
\begin{align*}
SF_q& =S^{\prec} & SK_q&
=S^{\succ} & SF_{q+1}^T&
=S_{\rhd} & SK_{q+1}^T&=
S_{\lhd} & SJ_q&
=S_{\lhd}^{\prec} \\ F_{p+1}S&
=S_{\triangledown} & K_{p+1}S&
=S_{\vartriangle} & F_p^TS&
=S^{\curlywedge} & K^T_pS&
=S^{\curlyvee} & J_pS&
=S^{\curlyvee}_{\triangledown}
\end{align*}

\begin{proof}[Proof of Theorem \ref{t3.1}.]
By Theorem \ref{t2.1}, we must
prove that for every $M,N\in
{\mathbb F}^{\,m\times n}$ there
exist $S\in {\mathbb
F}^{\,m\times m}$ and $R\in
{\mathbb F}^{\,n\times n}$ such
that
\begin{equation}       \label{3.2}
(M,N)+(SA-AR,\,SB-BR)=(P,Q),
\end{equation}
 where $(P,Q)$ is obtained from $({\cal A},{\cal B})-(A,B)$ by replacing the stars with elements of ${\mathbb F}$ and is uniquely determined by $(M,N)$.
The matrices $A$ and $B$ have
the block-diagonal form:
$A=A_1\oplus A_2\oplus \cdots,$
$B=B_1\oplus B_2 \oplus\cdots,$
where ${\cal P}_i=(A_i,\,B_i)$
are direct summands of the form
\eqref{3.1}. We apply the same
partition into blocks to $M$ and
$N$ and rewrite the equality
\eqref{3.2} for blocks:
\begin{equation*}
(M_{ij},N_{ij})+(S_{ij}A_j-A_iR_{ij},\,
S_{ij}B_j-B_iR_{ij})=(P_{ij},Q_{ij}),
\end{equation*}

Therefore, for every pair of
summands ${\cal P}_i
=(A_i,\,B_i)$ and ${\cal
P}_j=(A_j,\,B_j)$, $i\le j$, we
must prove that

(a) the pair $(M_{ij},N_{ij})$
can be reduced to the pair
$(P_{ij},Q_{ij})$ by
transformations $(M_{ij},N_{ij})
\mapsto
(M_{ij},N_{ij})+(\bigtriangleup
M_{ij},\bigtriangleup N_{ij})$,
where $$ \bigtriangleup M_{ij}:=
SA_j-A_iR,\quad \bigtriangleup
N_{ij}:=SB_j-B_iR $$
 with arbitrary $R$ and $S$; moreover, $(P_{ij},Q_{ij})$ is uniquely determined (more exactly, its entries on the places of stars are uniquely determined) by $(M_{ij},N_{ij})$; and, if $i<j$,

(b) the pair $(M_{ji},N_{ji})$
can be reduced to the pair
$(P_{ji},Q_{ji})$ by
transformations $(M_{ji},N_{ji})
\mapsto
(M_{ji},N_{ji})+(\bigtriangleup
M_{ji},\bigtriangleup N_{ji})$,
where $$ \bigtriangleup M_{ji}:=
SA_i-A_jR, \quad \bigtriangleup
N_{ji}:=SB_i-B_jR $$
 with arbitrary $R$ and $S$; moreover, $(P_{ji},Q_{ji})$ is uniquely determined by $(M_{ji},N_{ji})$.

\begin{description}
\item[\it Case 1: ${\cal P}_i=(F_p,K_p)$ and ${\cal
P}_j=(F_q,K_q)$, $p\le q$.]
${}$\nopagebreak

$\qquad$ (a) We have
$\bigtriangleup M_{ij}=
SF_q-F_pR=S^{\prec}-R_{\triangledown}$.
Adding $\bigtriangleup M_{ij}$,
we make $M_{ij}=0$; to preserve
it, we must further take $S$ and
$R$ for which $\bigtriangleup
M_{ij}=0$, i.e.
$S=[R_{\triangledown}\,\vdots\,]$,
where the points denote an
arbitrary column. Further,
$\bigtriangleup N_{ij}=SK_q-K_pR
=S^{\succ}-R_{\vartriangle}
=[R_{\triangledown}\,\vdots\,]^{\succ}-
R_{\vartriangle}
=[X_{\triangledown}\,\vdots\,]
-[\,\vdots\, X]_{\vartriangle}$,
where $X:=R^{\succ}$. Clearly,
$\bigtriangleup N_{ij}$ is an
arbitrary matrix
$[\delta_{\alpha\beta}]$ that
satisfies the condition: if its
diagonal
$D_t=\{\delta_{\alpha\beta}\,|\,\alpha-\beta
=t\}$ contains an entry from the
first row and does not contain
an entry from the last column,
then the sum of entries of this
diagonal is equal to zero.
Adding $\bigtriangleup N_{ij}$,
we make $N_{ij}=Z$, where $Z$ is
of the form \eqref{3.1a} but
with elements of ${\mathbb F}$
instead of the stars. If $i=j$,
then $p=q$, $N_{ii}=Z$ has size
$p\times (p-1)$, so $N_{ii}=0$
(see \eqref{3.1a}).

$\qquad$ (b) We have
$\bigtriangleup M_{ji}=
SF_p-F_qR$ and $\bigtriangleup
N_{ji}= SK_p-K_qR$; so we
analogously make $M_{ji}=0$ and
$N_{ji}=Z$. But since $Z$ has
size $q\times (p-1)$ and $p\le
q$, $N_{ji}=Z=0$ (see
\eqref{3.1a}).

\item[\it Case 2: ${\cal P}_i=(F_p,K_p)$ and ${\cal
P}_j=(I,J^{\mathbb
F}_q(\lambda))$.] ${}$

$\qquad$ (a) We have
$\bigtriangleup M_{ij}=
S-F_pR=S-R_{\triangledown}$.
Make $M_{ij}=0$; to preserve it,
we must further take
$S=R_{\triangledown}$. Then
$\bigtriangleup
N_{ij}=SJ_q^{\mathbb
F}(\lambda)-K_pR = (R
J_q^{\mathbb F}(\lambda))
_{\triangledown}-R_{\vartriangle}$.
Using the last row of $R$, we
make the last row of $N_{ij}$
equaling zero, then the next to
the last row equaling zero, and
so on util reduce $N_{ij}$ to
the form $ 0^{\uparrow}$ (with
elements of ${\mathbb F}$
instead of the stars).

$\qquad$ (b) We have
$\bigtriangleup M_{ji}=
SF_p-R=S^{\prec}-R$. Make
$\bigtriangleup M_{ji}=0$, then
$R= S^{\prec}$; $\bigtriangleup
N_{ji}= SK_p-J_q^{\mathbb
F}(\lambda)R=
S^{\succ}-(J_q^{\mathbb
F}(\lambda)S)^{\prec}$. We make
$N_{ji}=0$ starting with the
last row (with the last
horizontal strip if ${\mathbb
F}={\mathbb R}$ and
$\lambda\notin {\mathbb R}$).

\item[\it Case 3: ${\cal P}_i=(F_p,K_p)$ and ${\cal
P}_j=(J_q,I)$.] ${}$

$\qquad$ (a) We have
$\bigtriangleup N_{ij}= S-K_pR$,
make $N_{ij}=0$, then $S= K_pR=
R_{\vartriangle}$;
$\bigtriangleup M_{ij}=
SJ_q-F_pR =
(RJ_q)_{\vartriangle}-R_{\triangledown}.$
Reduce $M_{ij}$ to the form
$0^{\downarrow}$ starting with
the first row.

$\qquad$ (b) We have
$\bigtriangleup N_{ji}= SK_p-R$,
make $\bigtriangleup N_{ji}=0$,
then $R= SK_p=S^{\succ}$;
$\bigtriangleup M_{ji}=
SF_p-J_qR=S^{\prec}-
(J_qS)^{\succ}$. We make
$M_{ji}=0$ starting with the
last row.

\item[\it Case 4: ${\cal P}_i=(F_p,K_p)$ and ${\cal
P}_j=(F_q^T,K_q^T)$.] ${}$

$\qquad$ (a) We have
$\bigtriangleup M_{ij}=
SF_q^T-F_pR=S_{\rhd}-R_{\triangledown}$.
Reduce $M_{ij}$ to the form
$0^{\rightarrow}$, then
$(S_{\rhd}-R_{\triangledown})^{\prec}=
S-R^{\prec}_{\triangledown}=0$,
$S= R^{\prec}_{\triangledown}$.
Put $X:=R^{\prec}$, then
$S=X_{\triangledown}$ and
$R=[X\,\vdots\,]$, where the
points denote an arbitrary row.
Further, $\bigtriangleup
N_{ij}=SK_q^T-K_pR =
S_{\lhd}-R_{\vartriangle}=
(X_{\triangledown})_{\lhd}-[X\,\vdots\,]_{\vartriangle}$.
Clearly, $\bigtriangleup N_{ij}$
is an arbitrary matrix
$[\delta_{\alpha\beta}]$ that
satisfies the condition: if its
secondary diagonal
$D_t=\{\delta_{\alpha\beta}\,|\,\alpha+\beta
=t\}$ contains an entry from the
first row, then the sum of
entries of this diagonal is
equal to zero. Adding
$\bigtriangleup N_{ij}$, we
reduce $N_{ij}$ to the form $
0^{\uparrow}$.

$\qquad$ (b) We have
$\bigtriangleup M_{ji}=
SF_p-F_q^TR=S^{\prec}-R^{\curlywedge}$.
Make $M_{ji}=0$, then
$S=[R^{\curlywedge}\,\vdots\,]$.
Further, $\bigtriangleup N_{ji}=
SK_p-K_q^TR=S^{\succ}-
R^{\curlyvee}=[R^{\curlywedge}\,\vdots\,]^{\succ}-
R^{\curlyvee}$, make $N_{ji}=0$
starting with the last column.

\item[\it Case 5: ${\cal P}_i=(I, J_p^{\mathbb F}(\lambda))$ and ${\cal P}_j=(I,J_q^{\mathbb F} (\mu))$.]
${}$

$\qquad$ (a) We have
$\bigtriangleup M_{ij}= S- R$.
Make $M_{ij}=0$, then $S=R$;
$\bigtriangleup
N_{ij}=SJ_q^{\mathbb
F}(\mu)-J_p^{\mathbb F}
(\lambda)R$. Using Lemma
\ref{l5.1}, we make $N_{ij}=0$
if $\lambda\ne \mu$ and
$N_{ij}=H$ if $\lambda= \mu$.

$\qquad$ (b) We have
$\bigtriangleup M_{ji}= S- R$
and $\bigtriangleup
N_{ji}=SJ_p^{\mathbb
F}(\lambda)-J_q^{\mathbb
F}(\mu)R$. As in Case 5(a), make
$M_{ji}=0$, $N_{ji}=0$ if
$\lambda\ne \mu$ and $N_{ji}=H$
if $\lambda= \mu$.

\item[\it Case 6: ${\cal P}_i=(I, J_p^{\mathbb F} (\lambda))$ and ${\cal P}_j=(J_q,I)$.]
${}$

$\qquad$ (a) We have
$\bigtriangleup M_{ij}= SJ_q-
R=S^{\prec}_{\lhd}-R$. Make
$M_{ij}=0$, then $R=
S^{\prec}_{\lhd}$;
$\bigtriangleup
N_{ij}=S-J_p^{\mathbb
F}(\lambda)R =S-( J_p^{\mathbb
F}(\lambda)S)^{\prec}_{\lhd}$.
We make $N_{ij}=0$ starting with
the first column.

$\qquad$ (b) We have
$\bigtriangleup M_{ji}= S-J_q
R$, make $M_{ji}=0$, then
$S=R^{\curlyvee}_{\triangledown}$;
$\bigtriangleup
N_{ji}=SJ_p^{\mathbb F}
(\lambda)-R=(RJ_p^{\mathbb F}
(\lambda))^{\curlyvee}_{\triangledown}-R$.
We make $N_{ji}=0$ starting with
the last row.

\item[\it Case 7: ${\cal P}_i=(I, J_p^{\mathbb F} (\lambda))$ and ${\cal P}_j=(F_q^T,K_q^T)$.]
${}$

$\qquad$ (a) We have
$\bigtriangleup M_{ij}= SF_q^T-
R$. Make $M_{ij}=0$, then
$R=S_{\rhd}$; $\bigtriangleup
N_{ij}=SK_q^T-J_p^{\mathbb F}
(\lambda)R
=S_{\lhd}-(J_p^{\mathbb F}
(\lambda)S)_{\rhd}$. We reduce
$N_{ij}$ to the form
$0^{\leftarrow}$ starting with
the last row (with the last
horizontal strip if ${\mathbb
F}={\mathbb R}$ and
$\lambda\notin {\mathbb R}$).

$\qquad$ (b) We have
$\bigtriangleup M_{ji}= S-F_q^T
R$, make $M_{ji}=0$, then
$S=R^{\curlywedge}$,
$\bigtriangleup
N_{ji}=SJ_p^{\mathbb
F}(\lambda)-K^T_qR=
(RJ_p^{\mathbb
F}(\lambda))^{\curlywedge}-R^{\curlyvee}$.
We make $N_{ji}=0$ starting with
the first column (with the first
vertical strip if ${\mathbb
F}={\mathbb R}$ and
$\lambda\notin {\mathbb R}$).

\item[\it Case 8: ${\cal P}_i=(J_p,I)$ and ${\cal P}_j=(J_q,I)$.] ${}$ Interchanging the matrices in each pair, we reduce this case to Case 5.

\item[\it Case 9: ${\cal P}_i=(J_p,I)$ and ${\cal P}_j=(F_q^T,K_q^T)$.]
${}$

$\qquad$ (a) We have
$\bigtriangleup N_{ij}= SK_q^T-
R$. Make $N_{ij}=0$, then $R=
S_{\lhd}$; $\bigtriangleup
M_{ij}=SF^T_q-J_pR =S_{\rhd}-
(J_pS)_{\lhd}$. We reduce
$M_{ij}$ to the form
$0^{\rightarrow}$ starting with
the first column.

$\qquad$ (b) We have
$\bigtriangleup N_{ji}= S-K_q^T
R$, make $N_{ji}=0$, then
$S=R^{\curlyvee}$,
$\bigtriangleup
M_{ji}=SJ_p-F^T_qR=(RJ_p)^{\curlyvee}-R^{\curlywedge}$.
We make $M_{ji}=0$ starting with
the first column.

\item[\it Case 10: ${\cal P}_i=(F_p^T,K_p^T)$ and ${\cal P}_j=(F_q^T,K_q^T),\ p\ge q$.]
${}$

$\qquad$ (a) We have
$\bigtriangleup M_{ij}= SF_q^T-
F_p^TR$ and $\bigtriangleup
N_{ij}= SK_q^T- K_p^TR$, so
$(\bigtriangleup M_{ij})^T=
(-R^T)F_p- F_q(-S^T)$ and
$(\bigtriangleup N_{ij})^T=
(-R^T)K_p- K_q(-S^T)$. Reasoning
as in Case 1(a), we make
$M_{ij}^T=0$ and $N_{ij}^T=Z$,
that is $M_{ij} =0$ and
$N_{ij}=Z^T$ ($N_{ij}=0$ if
$i=j$).

$\qquad$ (b) We have
$\bigtriangleup M_{ji}=
SF_p^T-F_q^T R$ and
$\bigtriangleup N_{ji}=
SK_p^T-K_q^T R$, so we
analogously make $M_{ji}=0$ and
$N_{ji}=Z^T$. Since the size of
$Z^T$ is $(q-1)\times p$ and
$p\ge q$, by \eqref{3.1a} we
have $Z^T=0$.
\end{description}
\end{proof}

\section{Deformations of contragredient matrix pencils} \label{s4}

The canonical form problem for
pairs of matrices $A\in{\mathbb
F}^{\,m\times n},\ B\in{\mathbb
F}^{\,n\times m}$ under
transformations of
contragredient equivalence $$
(A,B)\mapsto
(SAR^{-1},RBS^{-1}),\quad
S\in{\rm GL}(m,\mathbb F),\ \
R\in{\rm GL}(n,\mathbb F), $$
(i.e., for representations of
the quiver \quivc\!\!\!) was
solved in \cite{pon, horn}: each
pair is uniquely, up to
permutation of cells
$J_r^{\mathbb F}(\lambda)$ in
$\oplus_i\Phi^{\mathbb
F}({\lambda_i})$, reduced to a
direct sum
\begin{multline}           \label{4.1}
(I, C)
\oplus\bigoplus_{j=1}^{t_1}
(I_{r_{1j}},J_{r_{1j}})
\oplus\bigoplus_{j=1}^{t_2}(J_{r_{2j}},I_{r_{2j}})\\
\oplus\bigoplus_{j=1}^{t_3}(F_{r_{3j}},G_{r_{3j}})
\oplus\bigoplus_{j=1}^{t_4}(G_{r_{4j}},F_{r_{4j}})
\end{multline}
(we use the notation
\eqref{3.1o} and put
$G_r:=K_r^T$), where $C$ is a nonsingular matrix of
the form \eqref{5.3''} and
$r_{i1}\ge r_{i2}\ge\dots\ge
r_{it_i}$.

\begin{theorem}   \label{t4.1}
One of the simplest miniversal
$\mathbb F$-deformations of the
canonical pair \eqref{4.1} under
contragredient equivalence over
$\mathbb F\in \{\mathbb C,
\mathbb R\}$ is the direct sum
of $(I, \tilde C)$ $(\tilde C$
is a simplest miniversal
$\mathbb F$-deformation of $C$
under similarity, see Theorem
\ref{t5.1}$)$ and {\em
$$\left(\left[
              \begin{tabular}{c|c|c}
                  $\oplus_j I_{r_{1j}}$&0&0\\       \hline
  0&$\oplus_j J_{r_{2j}}+\cal H$&$\cal H$\\       \hline
                  0&$\cal H$&
                              $\begin{matrix}
                   P_3&\cal H\\
                                 0&Q_4
                               \end{matrix}$
                 \end{tabular}\right],
          \left[\begin{tabular}{c|c|c}
  $\oplus_j J_{r_{1j}}+\cal H$&$\cal H$&$\cal H$\\    \hline
  $\cal H$&$\oplus_j I_{r_{2j}}$&0\\       \hline
  $\cal H$&0&$\begin{matrix}
                Q_3&0\\
                \cal H&P_4
              \end{matrix}$
                 \end{tabular}\right]
\right), $$} where $$
P_l=\left[\!\!\!\begin{tabular}{cccc}
$F_{r_{l1}}+H$&$H$&$\cdots$ &$H$
\\
 &$F_{r_{l2}}+H$&$ \ddots $&$\vdots$\\
    &&$\ddots $&$H$\\
         {\rm\rm\Large 0}&&& $F_{r_{lt_l}}+H$
         \end{tabular}\!\!\!\right],\
Q_l=\left[\!\!\!\begin{tabular}{cccc}
$G_{r_{l1}}$&&& {\rm\rm\Large 0} \\
      $H$&$ G_{r_{l2}}$&& \\
      $\vdots$&$\ddots$&$\ddots$&\\
      $H$&$\cdots$&$H$&$ G_{r_{lt_l}}$
         \end{tabular}\!\!\!\right]
$$ $ (l=3,\,4)$, $\cal H$ and
$H$ are matrices of the form
\eqref{5.0} and \eqref{5.1}, the
stars denote independent
parameters.
\end{theorem}

\begin{proof}
Let $(A,B)$ be the canonical
matrix pair \eqref{4.1} and let $({\cal A},{\cal B})$ be its deformation from Theorem \ref{t4.1}. By
Theorem \ref{t2.1}, we must
prove that for every $M\in
{\mathbb F}^{\,m\times n},\ N\in
{\mathbb F}^{\,n\times m}$ there
exist $S\in {\mathbb
F}^{\,m\times m}$ and $R\in
{\mathbb F}^{\,n\times n}$ such
that
\begin{equation*}  
(M,N)+(SA-AR,\,RB-BS)=(P,Q),
\end{equation*}
or, in the block form,
\begin{equation*}
(M_{ij},N_{ij})+(S_{ij}A_j-A_iR_{ij},\,
R_{ij}B_j-B_iS_{ij})=(P_{ij},Q_{ij}),
\end{equation*}
 where $(P,Q)$ is obtained from $({\cal A},{\cal B})-(A,B)$ by replacing its stars with complex numbers and is uniquely determined by $(M,N)$.

Therefore, for every pair of
summands ${\cal P}_i
=(A_i,\,B_i)$ and ${\cal
P}_j=(A_j,\,B_j)$, $i\le j$,
from the decomposition
\eqref{4.1}, we must prove that

(a) the pair $(M_{ij},N_{ij})$
can be reduced to the pair
$(P_{ij},Q_{ij})$ by
transformations $(M_{ij},N_{ij})
\mapsto
(M_{ij},N_{ij})+(\bigtriangleup
M_{ij},\bigtriangleup N_{ij})$,
where $$ \bigtriangleup M_{ij}=
SA_j-A_iR,\quad \bigtriangleup
N_{ij}=RB_j-B_iS $$
 with arbitrary $R$ and $S$; moreover, $(P_{ij},Q_{ij})$ is uniquely determined (more exactly, its entries on the places of stars are uniquely determined) by $(M_{ij},N_{ij})$; and, if $i<j$,

(b) the pair $(M_{ji},N_{ji})$
can be reduced to the pair
$(P_{ji},Q_{ji})$ by
transformations $(M_{ji},N_{ji})
\mapsto
(M_{ji},N_{ji})+(\bigtriangleup
M_{ji},\bigtriangleup N_{ji})$,
where $$ \bigtriangleup M_{ji}=
SA_i-A_jR,\quad \bigtriangleup
N_{ji}=RB_i-B_jS $$
 with arbitrary $R$ and $S$; moreover, $(P_{ji},Q_{ji})$ is uniquely determined by $(M_{ji},N_{ji})$.

\begin{description}

\item[\it Case 1: ${\cal P}_i=(I,J_p^{\mathbb F} (\lambda))$ and ${\cal P}_j=(I,J_q^{\mathbb F} (\mu))$.]
${}$

$\qquad$ (a) We have
$\bigtriangleup M_{ij}= S- R$.
Make $M_{ij}=0$, then $S=R$;
$\bigtriangleup
N_{ij}=RJ_q^{\mathbb F}
(\mu)-J_p^{\mathbb
F}(\lambda)S$. Using Lemma
\ref{l5.1}, we make $N_{ij}=0$
if $\lambda\ne \mu$, and
$N_{ij}=H$ (see \eqref{5.1}) if
$\lambda= \mu$.

$\qquad$ (b) We have
$\bigtriangleup M_{ji}= S- R$
and $\bigtriangleup
N_{ji}=RJ_p^{\mathbb F}
(\lambda)-J_q^{\mathbb F}
(\mu)S$. As in Case 1(a), make
$M_{ji}=0$, then $N_{ji}=0$ if
$\lambda\ne \mu$ and $N_{ji}=H$
if $\lambda= \mu$.

\item[\it Case 2: ${\cal P}_i=(I,J_p^{\mathbb F} (\lambda))$ and ${\cal P}_j=(J_q,I)$.]
${}$

 $\qquad$ (a) We have $\bigtriangleup M_{ij}= SJ_q-R$.  Make $M_{ij}=0$, then $R=SJ_q$, $\bigtriangleup N_{ij}= R-J_p^{\mathbb F}(\lambda)S=SJ_q-J_p^{\mathbb F} (\lambda)S $.
Using Lemma \ref{l5.1}, we make
$N_{ij}=0$ if $\lambda\ne 0$ and
$N_{ij}=H$ if $\lambda=0$.

$\qquad$ (b) We have
$\bigtriangleup M_{ji}= S-J_qR$.
Make $M_{ji}=0$, then $S= J_qR$,
$\bigtriangleup N_{ji}=
RJ_p^{\mathbb F} (\lambda)-S=
RJ_p^{\mathbb F} (\lambda)-
J_qR$. We make $N_{ji}=0$ if
$\lambda\ne 0$ and $N_{ji}=H$ if
$\lambda=0$.

\item[\it Case 3: ${\cal P}_i=(I,J_p^{\mathbb F} (\lambda)) $ and ${\cal P}_j=(F_q,G_q)$.]
${}$

 $\qquad$ (a) We have $\bigtriangleup M_{ij}= SF_q-R=S^{\prec}-R$. Make $M_{ij}=0$, then $R=S^{\prec}$, $\bigtriangleup N_{ij}= RG_q-J_p^{\mathbb F} (\lambda)S= S^{\prec}_{\lhd}-J_p^{\mathbb F} (\lambda)S= SJ_q-J_p^{\mathbb F} (\lambda)S$.
Using Lemma \ref{l5.1}, we make
$N_{ij}=0$ if $\lambda\ne 0$ and
$N_{ij}=H$ if $\lambda=0$.

$\qquad$ (b) We have
$\bigtriangleup M_{ji}=
S-F_qR=S-R_{\triangledown}$.
Make $M_{ji}=0$, then
$S=R_{\triangledown}$,
$\bigtriangleup N_{ji}=
RJ_p^{\mathbb F} (\lambda)-G_qS=
RJ_p^{\mathbb F}
(\lambda)-R^{\curlyvee}_{\triangledown}=
RJ_p^{\mathbb F}
(\lambda)-J_{q-1}R$. We make
$N_{ji}=0$ if $\lambda\ne 0$ and
$N_{ji}=H$ if $\lambda=0$.

\item[\it Case 4: ${\cal P}_i=(I,J_p^{\mathbb F} (\lambda)) $ and ${\cal P}_j=(G_q,F_q)$.]
${}$

 $\qquad$ (a) We have $\bigtriangleup M_{ij}= SG_q-R$. Make $M_{ij}=0$, then $R=S_{\lhd}$,
$\bigtriangleup N_{ij}=
RF_q-J_p^{\mathbb F} (\lambda)S=
S^{\prec}_{\lhd}-J_p^{\mathbb
F}(\lambda)S=
SJ_{q-1}-J_p^{\mathbb
F}(\lambda)S $. Using Lemma
\ref{l5.1}, we make $N_{ij}=0$
if $\lambda\ne 0$ and $N_{ij}=H$
if $\lambda=0$.

$\qquad$ (b) We have
$\bigtriangleup M_{ji}=
S-G_qR=S-R^{\curlyvee}$. Make
$M_{ji}=0$, then
$S=R^{\curlyvee}$,
$\bigtriangleup N_{ji}=
RJ_p^{\mathbb F} (\lambda)-F_qS=
RJ_p^{\mathbb F}
(\lambda)-R^{\curlyvee}_{\triangledown}=
RJ_p^{\mathbb F}
(\lambda)-J_qR$. We make
$N_{ji}=0$ if $\lambda\ne 0$ and
$N_{ji}=H$ if $\lambda=0$.

\item[\it Case 5: ${\cal P}_i=(J_p,I)$ and ${\cal P}_j=(J_q,I)$.] ${}$ Interchanging the matrices in each pair, we reduce this case to Case 1.

\item[\it Case 6: ${\cal P}_i=(J_p,I) $ and ${\cal P}_j=(F_q,G_q)$.] Interchanging the matrices in each pair, we reduce this case to Case 4.

\item[\it Case 7: ${\cal P}_i=(J_p,I) $ and ${\cal P}_j=(G_q,F_q)$.] Interchanging the matrices in each pair, we reduce this case to Case 3.

\item[\it Case 8: ${\cal P}_i=(F_p,G_p) $ and ${\cal P}_j=(F_q,G_q),\ i\le j$ (and hence $p\ge q$).] ${}$

 $\qquad$ (a) We have $\bigtriangleup N_{ij}=
RG_q-G_pS=R_{\lhd}-S^{\curlyvee}$.
Make $N_{ij}=0$, then
$R_{\lhd}=S^{\curlyvee}$.
Further, $\bigtriangleup M_{ij}=
SF_q-F_pR=
S^{\prec}-R_{\triangledown}$, so
$(\bigtriangleup
M_{ij})^{\curlyvee}=(S^{\curlyvee})^{\prec}-
R^{\curlyvee}_{\triangledown}=R^{\prec}_{\lhd}-R^{\curlyvee}_{\triangledown}=RJ_{q-1}-J_{p-1}R$
and the first row of
$\bigtriangleup M_{ij}$ is
arbitrary (due to the first row
of $S$). We make the first row
of $M_{ij}$ equaling zero.
Following the proof of Lemma
\ref{l5.1} and taking into
account that $p\ge q$, we make
all entries of the $(p-1)\times
(q-1)$ matrix
$M_{ij}^{\curlyvee}$ equaling
zero except for the last row and
obtain $M_{ij}=H$.

$\qquad$ (b) We have $i<j$,
$\bigtriangleup
M_{ji}=SF_p-F_qR=
S^{\prec}-R_{\triangledown}$.
Make $M_{ji}=0$, then
$S^{\prec}= R_{\triangledown}$.
Further, $\bigtriangleup N_{ji}=
RG_p-G_qS=R_{\lhd}-S^{\curlyvee},\
(\bigtriangleup
N_{ji})^{\prec}=R^{\prec}_{\lhd}-R^{\curlyvee}_{\triangledown}=RJ_{p-1}-J_{q-1}R$
and the last column of
$\bigtriangleup N_{ji}$ is
arbitrary (due to the last
column of $S$).
 We make the last column of $\bigtriangleup N_{ji}$ equaling zero. By Lemma \ref{l5.1} and the inequality $p\ge q$, we make all entries of the $(q-1)\times (p-1)$ matrix $N_{ji}^{\prec}$ equaling zero except for the first column and obtain $N_{ji}=H$.

\item[\it Case 9: ${\cal P}_i=(F_p,G_p) $ and ${\cal P}_j=(G_q,F_q)$.] ${}$

 $\qquad$ (a) We have $\bigtriangleup N_{ij}=
RF_q-G_pS=R^{\prec}-S^{\curlyvee}$.
Make $N_{ij}=0$, then
$R^{\prec}=S^{\curlyvee}$, i.e.
$R=X^{\curlyvee}$ and
$S=X^{\prec}$ for an arbitrary
$X$. Further, $\bigtriangleup
M_{ij}=
SG_q-F_pR=S_{\lhd}-R_{\triangledown}=X^{\prec}_{\lhd}-X^{\curlyvee}_{\triangledown}$,
we make $M_{ij}=H$.

$\qquad$ (b) We have
$\bigtriangleup M_{ji}=
SF_p-G_qR$ and $\bigtriangleup
N_{ji}= RG_p-F_qS.$ So we
analogously make $M_{ji}=0$ and
$N_{ji}=H$.

\item[\it Case 10: ${\cal P}_i=(G_p,F_p) $ and ${\cal P}_j=(G_q,F_q),\ i\le j$.] ${}$ Interchanging the matrices in each pair, we reduce this case to Case 8.
\end{description}
\end{proof}

\end{document}